\documentclass[11pt, reqno]{amsart}
\usepackage{amsfonts}
\usepackage{amscd, amssymb, amsmath, amsthm}
\usepackage{geometry}
\usepackage{hyperref}
\usepackage{enumerate}

\usepackage[numbers, sort&compress]{natbib}

\newtheorem{theo}{Theorem}[section]
\newtheorem*{theoA}{Theorem A}
\newtheorem*{theoB}{Theorem B}
\newtheorem*{theoC}{Theorem C}
\newtheorem*{theoD}{Theorem D}
\newtheorem{cor}[theo]{Corollary}
\newtheorem{dfn}[theo]{Definition}

\newtheorem{lma}[theo]{Lemma}

\newtheorem{prop}[theo]{Proposition}

\newtheorem{rmk}[theo]{Remark}
\newtheorem*{rk}{Remark}

\newcommand{\bt}{\begin{theo}}
\newcommand{\et}{\end{theo}}
\newcommand{\be}{\begin{equation}}
\newcommand{\ee}{\end{equation}}
\newcommand{\br}{\begin{rmk}}
\newcommand{\er}{\end{rmk}}
\newcommand{\bc}{\begin{cor}}
\newcommand{\ec}{\end{cor}}
\newcommand{\bp}{\begin{prop}}
\newcommand{\ep}{\end{prop}}
\newcommand{\bl}{\begin{lma}}
\newcommand{\el}{\end{lma}}
\newcommand{\bd}{\begin{dfn}}
\newcommand{\ed}{\end{dfn}}

\newcommand{\s}{\sigma}

\newcommand{\ga}{\gamma}
\newcommand{\Ga}{\Gamma}

\newcommand{\la}{\lambda}
\newcommand{\La}{\Lambda}
\newcommand{\vp}{\varphi}

\newcommand{\ka}{\kappa}

\newcommand{\lea}{\langle}
\newcommand{\ria}{\rangle}
\newcommand{\ot}{\otimes}

\newcommand{\pc}{\cdot}

\newcommand{\nb}{\nabla}

\newcommand{\tr}{{\rm tr}}

\newcommand{\vol}{\ast 1}
\newcommand{\Ric}{{\rm Ric}}

\def\r{{\rm rk}}

\def\Cl{{\rm C}\ell(M)}
\def\ClS{{\rm C}\ell(S^n)}
\def\X{\mathfrak X }
\def\K{\mathcal K}

\newcommand{\R}{\mathbb{R}}

\newcommand{\Rc}{\mathcal{R}}
\numberwithin{equation}{section}

\newcommand{\nocontentsline}[3]{}
\newcommand{\tocless}[2]{\bgroup\let\addcontentsline=\nocontentsline#1{#2}\egroup}

\begin{document}

\title{Twistor sections of Dirac bundles}


\author[Cardona]{Sergio A. H. Cardona${}^\dagger$}
\address[$\dagger$]{Instituto de Matem\'aticas, Universidad Nacional Aut\'onoma de M\'exico, Oaxaca de Ju\'a\-rez, M\'e\-xi\-co.}
\curraddr{}
\email{sholguin@im.unam.mx}
\thanks{($\dagger$) {CONACYT-UNAM} Research fellow.}

\author[Sol\'orzano]{Pedro Sol\'orzano${}^\star$}
\address[$\star$]{Instituto de Matem\'aticas, Universidad Nacional Aut\'onoma de M\'exico, Oaxaca de Ju\'a\-rez, M\'e\-xi\-co.}
\curraddr{}
\email{pedro.solorzano@matem.unam.mx}
\thanks{($\star$) {CONACYT-UNAM} Research fellow.}

\author[T\'ellez]{Iv\'an T\'ellez${}^\ddag$}
\address[$\ddag$]{Facultad de Econom\'ia. Universidad Aut\'onoma de San Luis Potos\'i, San Luis Po\-to\-s\'i, M\'e\-xi\-co.}
\curraddr{}
\email{ivan.tellez@uaslp.mx}
\thanks{($\ddag$) This author was supported by FORDECYT Grant No. 265667.}

\begin{abstract}
A Dirac bundle is a euclidean bundle over a riemannian manifold $M$ which is a compatible left $\Cl$-module, together with a metric connection also compatible with the Clifford action in a natural way. We prove some vanishing theorems and introduce the twistor equation within this framework. In particular, we exhibit a characterization of  solutions for this equation in terms of the Dirac operator $D$ and a suitable Weitzenb\"ock-type curvature operator $\Rc$. Finally, we analyze the especial case of the Clifford bundle to prove existence of nontrivial solutions of the twistor equation on spheres.
\end{abstract}

\maketitle

\section{Introduction}
The main purpose of this paper is to study some basic properties of Dirac bundles; in particular, the existence of vanishing theorems and of solutions to some classically motivated equations such as the twistor equation or the Killing equation.  Dirac bundles are one possible generalization of spin bundles to the case when no spin structure is required. In the latter, the aforementioned equations have been extensively studied, see for example \cite{BFGK, Friedrich, Friedrich1, FriedPok, Habermann, Habermann1, KuRa, Lichnerowicz}. As mentioned in the abstract, a Dirac bundle is a euclidean bundle over a riemannian manifold $M$ which is a compatible left $\Cl$-module, together with a metric connection also compatible with the Clifford action in a particular sense. With the above data a natural Dirac operator $D$ can be defined. The definition of a Dirac bundle was introduced in \cite{Lawson}.  Therein, some of the properties of Dirac bundles are developed and classical vanishing theorems of Bochner and Lichnerowicz are proved using {\it a generalized Bochner identity}, where the corresponding curvature term $\Rc$ is constructed using the curvature of the Dirac bundle. We call $\Rc$ the {\it Weitzenb\"ock curvature operator} of the Dirac bundle.

One of the fundamental ideas of twistor theory is that the geometry of a real structure, such as space-time, may be studied using an associated complex space \cite{Penrose1}; regarding the geometry and properties of the real structure as derived from this complex space. This idea  has been applied to reformulate various physical concepts in terms of twistors \cite{Penrose0} and has been adapted and developed in mathematics by many authors. For example, Atiyah and Ward \cite{AW} used the Penrose twistor transform to give a correspondence between algebraic bundles on the complex projective 3-space and minimum action solutions for $SU(2)$ Yang-Mills fields in the euclidean 4-space. 

In \cite{Penrose2}, Penrose introduced the so-called {\it twistor equation} for certain $4$-manifolds. In \cite{AHS} Atiyah, Hitchin and Singer showed that if $X$ is a self-dual 4-manifold then the projective bundle of anti-self-dual spinors inherits the structure of a complex analytic 3-manifold using solutions to a suitable twistor equation.

The study of the space of twistor spinors---solutions to the twistor equation on a spin manifold---has been carried out since then. There are special solutions to this equation called {\it Killing spinors}. It has been proved \cite{Friedrich} that if a spin manifold $M$ carries a Killing spinor then $M$ is a locally-irreducible Einstein space. In fact, the spinor bundle over the sphere is trivialized by Killing spinors \cite{Bar}. Also, if $(M,g)$ is a compact connected riemannian spin manifold carrying a nontrivial twistor spinor there exists an Einstein metric $g^{*}$, conformally equivalent to $g$, such that the space of twistor spinors on $(M,g)$ coincides with the space of Killing spinors on $(M,g^{*})$. For a compendium of results and properties of twistor spinors see \cite{BFGK}.

We extend the notion of twistor spinor to Dirac bundles by studying the natural translation of the twistor equation to this setting. The main results of this paper are the following. 
\begin{theoA} If the Weitzenb\"ock curvature operator $\Rc$ of a Dirac bundle over a compact riemannian manifold is positive semi-definite, then every section $\s$ of the kernel of the Dirac operator $D$ is parallel and satisfies 
\[
\lea\Rc\s, \s\ria=0.
\]
Furthermore, if $\Rc$ is also positive definite at some point, then $D$ has trivial kernel. 
\end{theoA}
This is an elementary generalization of a Bochner-type vanishing theorem when no spin structure is assumed nor required. 
\begin{theoB}  
Let $(S,\nb,\lea\cdot , \cdot \ria)$ be a Dirac bundle over a compact riemannian $n$-manifold $M$, and $D$ and $\Rc$ the aforementioned Dirac and Weitzenb\"ock curvature operators, respectively. A section $\s$ of $S$ is a twistor section if and only if
\[
D^2\s = \frac{n}{n-1}\Rc\s.
\]
Moreover, in this case,
\[
\nb_XD\s = \frac{n}{n-2}\left[\frac{1}{n-1}X\pc \Rc-\frac{1}{2}\Ric_X\right]\s,
\]
where $\Ric_X:S\to S$ is the Ricci curvature operator on $S$.
\end{theoB}
This theorem is a generalization of a well-known result of Lichnerowicz  (see Proposition 2 in \cite{Lich2} p. 335) in the context of spin bundles. 
\begin{theoC}
On a Dirac bundle $S$ over a compact riemannian $n$-manifold  $M$ the following are equivalent:
\begin{enumerate}[\indent\rm(i)]
\item A section $\s$ of $S$ is a Killing section with constant $\la$.
\item A section $\s$ of $S$ satisfies $D\s=-n\la\s$ and $\Rc\s = \la^2n(n-1)\s$.
\end{enumerate}
\end{theoC}
We would like to emphasize that the direction (i) implies (ii) in Theorem C is well-known in the geometry of spin bundles (See \cite{BFGK}, pp. 30-31). The converse in that setting---to the best of our knowledge---is not in the literature. 
Notice that in particular Theorem C implies the equivalence between (i) and (ii) in the classical case of spin bundles.
\begin{theoD} 
The twistor sections of $\ClS$, $n\geq 2$, are given by
\[
\s = c_1+df_1+d^*(\ast f_2) + \ast c_2,
\]
with constants $c_1, c_2$ and functions $f_1, f_2: S^n\to\R$ such that $\Delta f_i =n f_i$.
\end{theoD} 

This characterizes the solutions to the twistor equation in this setting, thus exhibiting the existence of nontrivial solutions different from those of the classical spin case.  

Theorems A, B, C and D appear below as Theorems \ref{TrivialKer}, \ref{twistorchar}, \ref{equiv1}, and  \ref{cliffSntwistor} respectively. 

In the Appendix we analyze the Killing equation on Clifford bundles over surfaces. In light of Theorem \ref{Kill2D}, this equation is much more restrictive than the twistor equation.

\setcounter{tocdepth}{1}
\tableofcontents
\section{Notations and conventions}

Throughout this communication some familiarity with the theory of Clifford algebras and spin bundles is expected. Classical references for these topics can be found within \cite{BFGK, Friedrich, Lawson}.

From now on, $M$ denotes an $n$-dimensional riemannian manifold, $\Cl$ its Clifford bundle, and it is assumed that all associated bundles (tangent, cotangent, exterior, Clifford, etc.) are endowed with the corresponding naturally induced metrics $\lea\cdot , \cdot \ria$ and metric connections $\nb$. The corresponding pointwise norms is denoted by $|\cdot|$.

We also adhere to the following conventions. Tangent vectors and vector fields are denoted with Latin capital letters $X,Y$, etc. and the space of vector fields by $\X(M)$. Bundles over $M$ are denoted by $E$, $S$, etc., and their spaces of sections (resp. compactly-supported sections) are denoted by $\Gamma (E)$, etc. (resp. by $\Gamma_{\rm c}(E)$, etc.). Individual elements and sections are denoted by Greek letters $\sigma$, $\tau$ etc., except for elements and sections of the Clifford bundle $\Cl$, which ---within the context of Clifford multiplication--- are denoted by Latin letters $a$, $b$, etc. The fiber over a point $p$ in $M$ of a bundle $E$ is denoted by $E_p$. Likewise, to avoid possible ambiguities, elements of $E_p$ are sometimes denoted as $\s_p$. 
\section{Dirac Bundles}
In this section we review the notion of Dirac bundle introduced in \cite{Lawson}, mention some fundamental properties and prove some basic new ones.  A Dirac bundle is a generalization a spin bundle to the case when no spin structure is required. Hence, it is possible to define a Dirac operator on such bundles.  We also recall the definitions of the Weitzenb\"ock curvature operator ---prove some elementary properties of it--- and of the connection laplacian of a Dirac bundle, and review the generalized Bochner identity that relates them.
\begin{dfn}[{\bf Dirac bundle}\hspace{-4pt} ]\label{defDirac}
A bundle $S$ over $M$ is a Dirac bundle if it is a left $\Cl$--module, together with a euclidean metric $\lea\cdot , \cdot \ria$ and a metric connection $\nb$,  satisfying the following two conditions:
\begin{enumerate}[\indent\rm(i)]
\item\label{prop1Dirac} For each $X\in T_pM$ and $\s,\tau\in S_p$
\begin{equation}\label{prop1Diraceq}
\lea X\cdot \s, \tau\ria +\lea \s, X\cdot \tau\ria = 0;
\end{equation}
 \item\label{prop2Dirac} for all $a\in \Ga(\Cl)$ and $\s\in \Ga(S)$, 
\begin{equation}\label{prop2Diraceq}
\nb (a\cdot\s ) = (\nabla a) \cdot \s + a\cdot \nb\s.
\end{equation}
\end{enumerate}
\end{dfn}
Spin bundles are automatically Dirac bundles. Another immediate example of a Dirac bundle is the Clifford bundle of $M$ itself. Since as a vector bundle $\Cl$ is isomorphic to the exterior bundle (see \cite{Lawson} p. 10), 
\if  
Using this identification, In terms of an orthonormal basis $e_1,\ldots, e_n$ of $T_pM$, the Clifford multiplication is given by
\begin{equation}\label{Cliffordmult}
e_i\pc \psi = e_i\wedge\psi - \iota_{e_i} \psi\,,
\end{equation}
with $\psi\in\La^{*}(M)_p$ 
\fi
the existence of Dirac bundles is not topologically obstructed as in the case of the bundle of spinors on a spin manifold. Several of the algebraic properties of the Clifford action that are well known in the spin case also hold here, e.g. from \eqref{prop1Dirac} it follows that for $X\in T_pM$ and $\s\in S_p$ 
\be\label{Cliffnorm}
|X\cdot\s|=|X||\s|.
\ee
Another property of $\Cl$ that is required in the sequel is the following (see \cite{Friedrich} p. 6. for details): there exists a fiberwise-linear involution $\ga$ on $\Cl$ such that for all $X\in T_pM$
\be\label{invol1}
\ga_p(X)=X,
\ee 
and  for all $a,b\in \Cl_p$
\be\label{invol2}
\ga_p(a\pc b)=\ga_p(b)\pc\ga_p(a).
\ee
Notice that from \eqref{invol2} it is readily verified that for any $a\in \Ga(\Cl)$,
\be\label{gammanabla}
\nb(\ga(a))=\ga(\nb a).
\ee
where $\nb$ is the natural connection on $\Cl$.

Recall that a parallel subbundle of a euclidean bundle with connection is a subbundle that is invariant under parallel translations; equivalently, if for any section of the subbundle its covariant derivative is also a section of the subbundle (cf. \cite{KNvol1} p. 114). 

The following proposition establishes some elementary properties of Dirac bundles. 

\begin{prop}\label{trivDirac} Let $S$, $S_1$ and $S_2$ be Dirac bundles and $E$ a euclidean vector bundle over $M$. The following are Dirac bundles over $M$:
\begin{enumerate}[\indent\rm(i)]
\item\label{Diracprop1} Any parallel subbundle of $S$ which is also a left $\Cl-$submodule of $S$ with the restricted action from $S$.
\item\label{Diracprop2} The Whitney sum $S_1\oplus S_2$ of $S_1$ and $S_2$ with the action defined componentwise.
\item\label{Diracprop3} The dual $S^{*}$ of $S$ with the action given by
\begin{equation}\label{DualCliffAction}
(a\pc \s^{*})(\xi) = \s^{*}(\gamma(a)\pc\xi),
\end{equation}
for any  $\s^*\in S^*$ and $a\in\Cl$ over the same point and $\ga$ the aforementioned involution.
\item\label{Diracprop4} The tensor product $S\otimes E$ of $S$ and $E$ with the action given by
\be
a\cdot(\sigma\otimes \xi)=(a\cdot \s)\otimes\xi,
\ee
for any  $\s\in S$, $\xi\in E$ and $a\in\Cl$ over the same point.
\end{enumerate}
\end{prop}
\br 
 A Dirac subbundle of a Dirac bundle is any subbundle that satisfies the conditions in \eqref{Diracprop1}. Also, a Dirac bundle is irreducible if it has no nontrivial Dirac subbundles. 
 \er
\br If $E$ is itself a Dirac bundle, then $S\otimes E$ can be considered as a Dirac bundle in two distinct ways. This is the case, e.g., for the bundle ${\rm Hom}(S,S)$ of a Dirac bundle $S$.
\er

\begin{proof}[Proof of Proposition \ref{trivDirac}]
The proof of \eqref{Diracprop1} and \eqref{Diracprop2} is straightforward and \eqref{Diracprop4} is proved in \cite{Lawson} pp. 121-122 as Proposition 5.10 therein. To establish \eqref{Diracprop3}, observe that it is straightforward to verify that the action given by \eqref{DualCliffAction} induces a $\Cl-$module structure on $S^{*}$. Since $S$ is Dirac, \eqref{invol1}, \eqref{DualCliffAction}, and the definition of the metric dual yield that
\be\label{DualCliffiden}
X\pc \s^{*}= -(X\pc \s)^{*}
\ee
for $X\in T_pM$ and $\s\in S_p$. Now, \eqref{DualCliffiden} gives
\begin{align*}
\lea X\pc \s^{*}, \tau^{*}\ria &=-\lea (X\pc \s)^{*}, \tau^{*}\ria = -\lea X\pc \s, \tau\ria\\ &= \lea \s, X\pc \tau\ria = \lea \s^{*}, (X\pc \tau)^{*}\ria=-\lea \s^{*}, X\pc \tau^{*}\ria,
\end{align*}
 establishing \eqref{prop1Dirac} in Definition \ref{defDirac} for $S^*$.
 
On the other hand, the standard dual connection is given by
\be\label{dualconn}
(\nb_X \s^{*})(\tau)= X(\s^{*}(\tau))- \s^{*}(\nb_X\tau),
\ee
with $X\in T_pM$, $\tau\in \Gamma(S)$ and $\sigma^*\in\Gamma(S^*)$.  It follows from \eqref{DualCliffAction}, \eqref{dualconn} and \eqref{gammanabla} that
\begin{align*}
(\nb_X (a\pc \s^{*}))(\tau) &= X( (a\pc \s^{*})(\tau)) -(a\pc \s^{*})(\nb_X \tau)\\
&= X(\s^{*}(\ga(a)\pc \tau)) -\s^{*}(\ga(a)\pc \nb_X \tau)\\
&=X(\s^{*}(\ga(a)\pc \tau))+\s^{*}(\ga(\nb_X a)\pc \tau)\\ &\phantom{=}-\s^{*}(\ga(a)\pc\nb_X\s)-\s^{*}(\nb_X(\ga(a))\pc \tau)\\
&=\s^{*}(\ga(\nb_X a)\pc \tau)+ X(\s^{*}(\ga(a)\pc \tau)) -\s^{*}( \nb_X(\ga(a)\pc \tau))\\
&=\s^{*}(\ga(\nb_X a)\pc \tau)+ (\nb_X \s^{*})(\ga(a)\pc \tau))\\
&=((\nb_X a)\pc \s^{*}+a\pc\nb_X\s^{*})(\tau),
\end{align*}
establishing \eqref{prop2Dirac} in Definition \ref{defDirac} for $S^*$.  
\end{proof}
\bl\label{lemasub} The orthogonal complement of a Dirac subbundle is also a Dirac subbundle.
\el
\proof Let $E\subset S$ be a Dirac subbundle of $S$. The proof that $E^\perp$ is also parallel is straightforward.
To see that it is a $\Cl$-module as well, let $e_1,\ldots,e_n$ be a local orthonormal frame. By part \eqref{prop1Dirac} of Definition \ref{defDirac},
\[\lea e_{i_1}\cdot \ldots\cdot e_{i_k}\cdot\s, \tau \ria=(-1)^k\lea \sigma, \ga(e_{i_1}\cdot \ldots\cdot e_{i_k})\cdot\tau \ria,
\]
which implies that for any $a\in\Cl$,  $\s,\tau \in S$,  there exists $b\in\Cl$ such that 
\[
\lea a\cdot \s, \tau \ria=\lea \s, b\cdot \tau\ria.
\]
Letting $\s \in E^\perp$ and $\tau\in E$, shows that $a\cdot \s \in E^\perp$.
\endproof
\bp Dirac bundles of finite rank are the sum of irreducible Dirac bundles. 
\ep
\proof  The proof is by induction on the rank of $S$:  Let $E$ be a nontrivial irreducible Dirac subbundle of $S$, then by Lemma \ref{lemasub}, $E^{\perp}$ is a Dirac subbundle whose rank is strictly less than the rank of $S$.
\endproof
On a Dirac bundle, it is possible to define a Dirac operator imitating the classical definition of the Dirac operator of a spin bundle.  Consequently, the {\it Dirac operator} $D:\Gamma(S)\to \Gamma(S)$ of a  Dirac bundle is locally given by
\be
D\s = \sum_{i=1}^{n}e_{i}\pc \nb_{e_i}\s,
\ee
where $e_1,\ldots, e_n$ is a local orthonormal frame of $M$. It is known that $D$ and $D^2$ are elliptic operators. Observe that $D$ satisfies the following Leibniz rule:
\be\label{LeibD}
D(f\s)=\nb f\cdot \s+f D\s,
\ee
for $f$ and $\s$ smooth. The formula
\be\label{globalinnprod}
(\s, \tau) = \int_{M}\lea \s, \tau \ria
\ee
defines an inner product,\footnote{This and some subsequent constructions can be done on a euclidean vector bundle with a metric connection. Since the main purpose of this communication is to focus on Dirac bundles, we restrict our attention to them.} whose norm is denoted by $\|\cdot\|$, on the corresponding $L^2$-space. There, the Dirac operator $D$ is formally self-adjoint with respect to it, i.e.,
\be\label{Dselfad}
(D\s, \tau) = (\s, D\tau).
\ee
In particular, if $M$ is compact, $\ker D =\ker( D^2)$. An extension of this to complete riemannian manifolds is also known (see Theorem 5.7 in \cite{Lawson} p. 117). Here we present a different proof that is a consequence of the following result.
\begin{prop}\label{Friedineq}
Let $M$ be complete and let $S$ be a Dirac bundle over it.  Then, for each $\s\in \Ga(S)$ and $t>0$,
\be
\|D\s\|^2 \leq t\|D^2\s\|^2+\frac{1}{t}\|\s\|^2.
\ee
\end{prop}
\begin{proof} (The arguments below follow closely those in \cite{Friedrich} pp. 96-98.) Observe that there is nothing to prove if $\|\s\|$ is infinite; therefore assume that $\|\s\|$ is finite. 

Since $M$ is complete, in particular it is metrically complete with respect to the shortest path metric $d$ and the metric balls,
\[
B_p(r)=\{q\in M : d(p,q)<r\}
\] 
for each $p\in M$ and $r>0$, are pre-compact.  

Fix a point $p_0\in M$ and let $B(r)$ denote $B_{p_0}(r)$. Using a standard argument of bump functions it is possible to define, for any $r>0$, a smooth function $f_r:M\to [0,1]$ such that $f_r\equiv 1$ on $B(r)$, ${\rm supp} f_r\subset B(2r)$, and 
\be\label{fr0}
|\nb(f_r)|^2\leq \frac{C^2}{r^2},
\ee
for a certain constant $C$ independent of $r$ (see \cite{Friedrich} pp. 95-96 for details on this construction). For any $\varepsilon>0$ and $\s\in\Ga(S)$, we have 
\begin{align*}
\int_{B(2r+\varepsilon)}|f_rD\s|^2&=\int_{B(2r+\varepsilon)}\left\lea D(f_r^2D\s),\s \right\ria\\
&=\int_{B(2r+\varepsilon)}\lea 2f_r\nb(f_r)\cdot D\s,\s \ria+\int_{B(2r+\varepsilon)}\lea f_r^2D^2\s,\s \ria.
\end{align*}
To show this, observe that in the first line, integration over $B(2r+\varepsilon)$ coincides with integration over the whole $M$, and thus we can use \eqref{Dselfad}. Using \eqref{prop1Dirac} of Definition \ref{defDirac} in the above expression and letting $\varepsilon$ go to zero yields

\be\label{fr1}
\int_{B(2r)}|f_rD\s|^2=\int_{B(2r)}\lea D^2\s,f_r^2\s \ria-\int_{B(2r)}\lea f_r D\s,2\nb(f_r)\cdot\s \ria.
\ee
The fact that, for any $t>0$,
\[
|\lea x,y\ria|\leq \frac t2 |x|^2+\frac1{2t}|y|^2
\]
allows us to estimate each of the right-hand-side terms of \eqref{fr1}. Indeed, from the definition of $f_r$, it yields
\be\label{fr2}
\left|\int_{B(2r)}\lea D^2\s,f_r^2\s \ria\right|\leq \frac t2 \int_{B(2r)} |D^2\s|^2+ \frac 1{2t} \int_{B(2r)} |\s|^2.
\ee
Also, for $t=1$, using \eqref{Cliffnorm} and \eqref{fr0}, 
\be\label{fr3}
\left|\int_{B(2r)}\lea f_r D\s,2\nb(f_r)\cdot\s \ria\right|\leq \frac12\int_{B(2r)}|f_rD\s|^2+\frac{2C^2}{r^2}\int_{B(2r)}|\s|^2.
\ee
Applying the triangle inequality to \eqref{fr1}, together with \eqref{fr2} and \eqref{fr3} ---and grouping like terms--- yields 
\[
\int_{B(2r)}|f_rD\s|^2\leq t\int_{B(2r)} |D^2\s|^2+\left(\frac1t+ \frac{4C^2}{r^2}\right)\int_{B(2r)} |\s|^2.
\]
Finally, since
\[
\int_{B(r)}|D\s|^2\leq \int_{B(2r)}|f_rD\s|^2,
\]
letting $r$ go to infinity in the previous two inequalities finishes the proof.
\end{proof}
\bc Let $M$ be complete. In the space of $L^2$-section over $S$,
\be\label{kerD}
\ker D =\ker( D^2).
\ee
\ec
\proof
One inclusion is evident. Assume $D^2\s=0$ for an $L^2$-section $\s$ of $S$. Elementary elliptic theory yields that $\s\in\Ga(S)$ (see \cite{Lawson} pp. 113 and 193 for details). 
 From Proposition \ref{Friedineq}, 
\[
\|D\s\|^2 \leq\frac{1}{t}\|\s\|^2,
\]
for any $t>0$. Therefore $D\s=0$.
\endproof
The {\it connection laplacian} $\nb^{*}\nb:\Ga(S)\to\Ga(S)$ is defined by 
\[
\nb^{*}\nb\s= -\tr(\nb^2\s), 
\]
for any $\s\in\Ga(S)$, where for each $X, Y\in\X(M)$ the operator $\nb^2_{X,Y}:\Ga(S)\to\Ga(S)$ is given by $\nb^2_{X,Y}\s=\nb_X\nb_Y\s - \nb_{\nb_XY}\s$ (notice that it is tensorial in $X$ and $Y$). The name {\it laplacian} is justified since it satisfies
\begin{equation}\label{CLapSelfad}
(\nb^{*}\nb \s, \tau)= (\nb \s, \nb\tau)
\end{equation}
for all $\s, \tau \in\Ga_{\rm c}(S)$; in the right-hand side, $\nb\s$ and $\nb\tau$ are regarded as sections of $T^*M\otimes S$.  In terms of a local orthonormal frame $e_1,\ldots, e_n$, 
\begin{equation}\label{defconlap}
\nb^{*}\nb\s  = - \sum_{j=1}^n \nb^2_{e_j,e_j}\s,
\end{equation}
and 
\be\label{innprodnabla} 
(\nb \s, \nb\tau) = \sum_{j} (\nb_{e_j}\s, \nb_{e_j}\tau).
\ee
From \eqref{CLapSelfad} it is evident that the connection laplacian is nonnegative and formally self-adjoint. Moreover, in the compact case \eqref{CLapSelfad} implies that $\nb^{*}\nb \s =0$ if and only if $\s$ is parallel (details in \cite{Lawson} pp. 154-155). 

The {\it Weitzenb\"ock curvature operator} of a Dirac bundle $S$ over $M$ is the linear operator $\Rc:S\to S$ defined for a local orthonormal frame $e_1,\ldots, e_n$ by the formula
\begin{equation}\label{defR}
\Rc\s = \sum_{i<j} e_i\cdot e_j\cdot R(e_i, e_j)\s.
\end{equation}
This operator is well-defined and satisfies the generalized Bochner identity (see \cite{Lawson}, p. 155),
\begin{equation}\label{GenBochnerId}
D^2 = \nb^{*}\nb + \Rc.
\end{equation}
Notice that it is evident from the Bochner identity that $\Rc$ is formally self-adjoint, since both $D$ and $\nb^*\nb$ are too (cf. \eqref{Dselfad} and \eqref{CLapSelfad}). 
For $\Rc$ even more is true, as the following proposition shows.
\bp\label{Rcptselfad} The Weitzenb\"ock operator of a Dirac bundle is pointwise self-adjoint.
\ep
 \proof  
Let $e_1,\ldots, e_n$ be a local orthonormal frame.  It is easy to verify that 
\[
R(e_i,e_j)(e_k\cdot\sigma)=e_k\cdot R(e_i,e_j)\sigma.
\]
Using this and the standard fact that  $\lea R(e_i,e_j)\s,\tau \ria=-\lea \s, R(e_i,e_j)\tau \ria$,
 \begin{align*} 
 \lea \Rc\s,\tau\ria &= \sum_{i<j} \lea e_i\cdot e_j\cdot R(e_i,e_j)\sigma,\tau \ria \\
 &=\sum_{i<j} \lea \s, R(e_i,e_j)(e_i\cdot e_j\cdot \tau) \ria \\
 &= \sum_{i<j} \lea \s, e_i\cdot e_j\cdot R(e_i,e_j)\tau \ria \\
 \pushQED{\qed} 
 &=  \lea\s, \Rc\tau\ria.\qedhere
\popQED
 \end{align*}
The next proposition provides specific formulas for the Weitzenb\"ock operator of the bundles described in Proposition \ref{trivDirac}. 
\bp\label{BDprops} Let $\Rc$ denote the Weitzenb\"ock operator on each of the following cases.
\begin{enumerate}[\indent\rm(i)]
\item\label{BD1} On a Dirac subbundle $\Rc$ is the restriction of the ambient Weitzenb\"ock operator. 
\item\label{BD2} On the Whitney sum of two Dirac bundles, 
\be\label{BDWhit}
 \Rc (\xi,\zeta)=(\Rc \xi, \Rc \zeta).
\ee
\item\label{BD3} On the dual of a Dirac bundle,  
\be\label{BDdual}
\Rc(\s^*)=(\Rc \s)^*.
\ee
\item\label{BD4} On a tensor product of a Dirac bundle and a euclidean bundle, locally, 
\be\label{BDTens}
\Rc(\s\otimes\xi)=\mathcal{R}\s\ot \xi + \sum_{i<j}e_i\cdot e_j\cdot \s\ot R(e_i,e_j)\xi,
\ee
where $e_1,\ldots,e_n$ is a local orthonormal frame.
\end{enumerate}
\ep
\begin{rk} The last term of the right-hand side of \eqref{BDTens} is already defined in \cite{Lawson} p. 164. It could be described invariantly as $\tr(\omega\mapsto \omega\cdot \s\otimes R(\omega)\xi)$, with $\omega\in\Lambda^2TM$.  Proposition \ref{TensorPos} below uses a real-valued version of this trace.
\end{rk}
\proof[Proof of Proposition \ref{BDprops}] Parts (i) and (ii) are immediate.  To see parts (iii) and (iv), let $e_1,\ldots,e_n$ be a local orthonormal frame. For the dual bundle,  let $\s, \eta \in S$ over the same point. Then, since $R(X,Y)(\s^*)=(R(X,Y)\s)^*$,
\begin{align*}
[\Rc (\s^*)](\eta)&=\sum_{i<j}e_i\cdot e_j \cdot R(e_i,e_j)(\s^*)(\eta)=\sum_{i<j}e_i\cdot e_j \cdot (R(e_i,e_j)\s)^*(\eta)\\
&=\sum_{i<j}\lea R(e_i,e_j)\s, \gamma(e_j)\cdot \gamma(e_i)\cdot \eta\ria=\sum_{i<j}\lea e_i\cdot e_j\cdot R(e_i,e_j)\s, \eta\ria\\
&=\lea \Rc \s, \eta\ria= (\Rc \s)^*(\eta).
\end{align*}
For the tensor product, let $\s, \xi \in S$ over the same point. Then
\begin{align*}
\mathcal{R}(\s\ot \xi) &=\sum_{i<j}e_i\cdot e_j\cdot R(e_i, e_j)(\s\ot \xi)\\
&= \sum_{i<j}e_i\cdot e_j\cdot (R(e_i, e_j)\s\ot \xi+ \s\ot R(e_i, e_j)\xi)\\
\pushQED{\qed} 
&=\mathcal{R}\s\ot \xi + \sum_{i<j}e_i\cdot e_j\cdot \s\ot R(e_i,e_j)\xi.\qedhere
\popQED
\end{align*} 
\section{Vanishing results}

The purpose of this section is to prove Theorem A as well as to establish some basic properties related to the positivity\footnote{Recall that a self-adjoint operator $\mathcal P$ is positive semi-definite (resp. definite) if for all $v$ (resp. for $v\neq0$) in its domain $\lea \mathcal P v, v\ria\geq0$ (resp. $\lea \mathcal P v, v\ria>0$).} of the Weitzenb\"ock operator.  As a general principle, Bochner-type identities gives rise to Bochner vanishing theorems, e.g. in Complex Geometry there is a Bochner identity relating the Chern connection and  the mean curvature for any holomorphic vector bundle. Such a result together with some conditions on the mean curvature implies the parallelism or nonexistence of holomorphic sections (see \cite{Kobayashi} ch. 3). In our context this is also the case: the Bochner identity \eqref{GenBochnerId} together with a positivity condition implies Theorem A, which is a particular type of Bochner vanishing theorem. 
\begin{theo}\label{TrivialKer} If the Weitzenb\"ock curvature operator $\Rc$ of a Dirac bundle over a compact manifold is positive semi-definite, then every section $\s$ of the kernel of the Dirac operator $D$ is parallel and satisfies 
\be\label{Rcperp}
\lea\Rc\s, \s\ria=0.
\ee
Furthermore, if $\Rc$ is positive definite at some point, then $D$ has trivial kernel. 
\end{theo}
\begin{proof}  To see the first claim, assume $\Rc$ is positive semi-definite and let $\s$ be a section in the kernel of $D$. Then $\lea\Rc\s, \s\ria\geq0$ and by \eqref{CLapSelfad} and \eqref{GenBochnerId},
\[
0=(D^2\s,\s)=(\nb^{*}\nb \s, \s) +(\Rc \s,\s)=\|\nb \s\|^2+ (\Rc \s,\s).
\]
Being nonnegative, the two terms on the right-hand side of the last equation must vanish, therefore $\s$ is parallel and satisfies \eqref{Rcperp}.

Furthermore, assume $\Rc$ is also positive definite at a point. Therefore it is positive definite in a neighborhood and $(\Rc\s,\s)>0$ for any $\s$ that doesn't vanish identically on that neighborhood. Now let $\s$ be a section in the kernel of $D$.  By the first part it is parallel and $(\Rc\s, \s)=0$. Being parallel, if $\s$ is nonzero at a point, it is nonzero everywhere. Thus, $\s$ is identically zero.
\end{proof}
Next we study some further properties of the Weitzenb\"ock operator and establish conditions under which its positivity is preserved under the bundle constructions considered in Proposition \ref{trivDirac}. The Weitzenb\"ock operator on any Dirac subbundle inherits any positivity from the ambient. For the Whitney sum and the dual we have the following results. 
\bp
The Weitzenb\"ock operator on the Whitney sum of Dirac bundles is positive semi-definite (resp. definite) if and only if the Weitzenb\"ock operator on each summand is positive semi-definite (resp. definite). 
\ep
\proof Both directions follow readily from \eqref{BDWhit}.\endproof
\bp The Weitzenb\"ock operator on a Dirac bundle is positive semi-definite (resp. definite) if and only if the Weitzenb\"ock operator on its dual is positive semi-definite (resp. definite). 
\ep
\proof Both directions follow readily from \eqref{BDdual}.\endproof
The case of the tensor product requires a more detailed analysis.
\bl\label{Rgaritenslema}
Let $S$ and $E$ be euclidean bundles over $M$ and assume further that $S$ is Dirac. The Weitzenb\"ock operator on $S\ot E$ satisfies that
\be\label{Rgaritensor}
\lea \mathcal{R}(\s\ot \xi), \s\ot \xi \ria =\lea\Rc \s, \s \ria \lea \xi, \xi\ria
\ee
for any $\s\in S$ and $\xi\in E$ over the same point.
\el
\begin{proof}
From \eqref{BDTens} it follows that
\begin{align}
\lea \mathcal{R}(\s\ot \xi), \tau\ot \zeta \ria &= \lea \Rc \s\ot \xi, \tau\ot \zeta\ria + \sum_{i<j} \left\lea e_i\cdot e_j\cdot \s\ot R(e_i,e_j)\xi, \tau\ot \zeta \right\ria\nonumber\\
&= \lea\Rc \s, \tau \ria \lea \xi, \zeta\ria + \sum_{i<j} \lea e_i\cdot e_j\cdot \s,\tau\ria \lea R(e_i,e_j)\xi , \zeta\ria.\label{Rctensorprop}
\end{align}
Taking $\tau\ot \zeta=\s\ot \xi$ finishes the proof.
\end{proof}
\bp
Let $S$ and $E$ be euclidean bundles over $M$ and assume further that $S$ is Dirac. The  Weitzenb\"ock operator on $S$ is positive semi-definite (resp. definite) if the Weitzenb\"ock operator on $S\ot E$ is also positive semi-definite (resp. definite). 
\ep
\proof
If the Weitzenb\"ock operator on $S \ot E$ is positive semi-definite then $\lea\Rc\eta, \eta\ria\geq0$ (resp. $\lea\Rc\eta, \eta\ria>0$) for any $\eta\in S\ot E$ (resp. for $\eta\neq0$). In particular, by letting $\eta=\sigma\ot \xi$ the result now follows from \eqref{Rgaritensor}.
\endproof
It is important to mention that in other contexts these types of definiteness are inherited in tensor products (see for instance \cite{Kobayashi} p. 53); to accomplish that, both factors must possess such property. For Dirac bundles, even though the Dirac structure on $S\ot E$ is defined only using the Dirac structure of $S$, \eqref{Rctensorprop} shows that the Weitzenb\"ock operator on $S\ot E$ depends on the geometry of $E$.  A condition for such converse to hold is now given. 
\bp\label{TensorPos} Let $S$ be a Dirac bundle with positive Weitzenb\"ock operator and let $E$ be any euclidean bundle. Let $\Theta$ be the bilinear form on $S\otimes E$ determined by 
\be\label{Tensorform}
\Theta(\s\otimes\xi, \tau\otimes\zeta)=\tr\bigg(\omega\mapsto \lea \omega\cdot \s,\tau\ria \lea R(\omega)\xi , \zeta\ria\bigg).
\ee
The Weitzenb\"ock operator on $S\otimes E$ is positive semi-definite if and only if
\be\label{posDir}
\lea\Rc \s, \s \ria\lea\Rc \tau, \tau \ria|\xi|^2|\zeta|^2-\lea\Rc \s, \tau \ria^2\lea \xi, \zeta \ria^2-2\lea\Rc \s, \tau \ria\Theta-\Theta^2\geq0,
\ee
for all $\s\otimes\xi, \tau\otimes\zeta\in S\otimes E$ over the same point. Furthermore, it is positive definite if the inequality is strict for nonzero vectors. 
\ep
\proof It follows from \eqref{Rctensorprop} and \eqref{Tensorform} that 
\[
\lea \mathcal{R}(\s\ot \xi),  \tau\otimes\zeta \ria =\lea\Rc \s, \tau \ria \lea \xi, \zeta \ria+\Theta(\s\otimes\xi, \tau\otimes\zeta).
\]
Notice that $\Theta(\eta,\eta)=0$ for any $\eta\in S\ot E$ (cf. \eqref{Rctensorprop} in Lemma \ref{Rgaritenslema}). For linearly independent $\s\otimes\xi$ and $\tau\otimes\zeta$ in $S\otimes E$,
\[
\begin{bmatrix}
\lea\Rc \s, \s \ria |\xi|^2& \lea\Rc \s, \tau \ria\lea \xi, \zeta \ria+\Theta\\
\lea\Rc \s, \tau \ria\lea \xi, \zeta \ria+\Theta& \lea\Rc \tau, \tau \ria |\zeta|^2
\end{bmatrix}
\]
is the matrix associated to $\Rc$ restricted to the subspace generated by $\s\otimes\xi$ and $\tau\otimes\zeta$.
Observe that the left-hand side of \eqref{posDir} is nothing but the determinant of this matrix; from which the conclusion follows.
\endproof
An immediate consequence is the following result.
\bp Let $S$ be a Dirac bundle with positive semi-definite (resp. definite) Weitzenb\"ock operator and let $E$ be a euclidean bundle with connection.  If $\Theta\equiv0$ then the Weitzenb\"ock operator on $S\otimes E$ is positive semi-definite (resp. definite).
\ep
\proof If $\Theta\equiv0$, then all that is left to verify \eqref{posDir} is that  
\[\lea\Rc \s, \s \ria\lea\Rc \tau, \tau \ria|\xi|^2|\zeta|^2\geq\lea\Rc \s, \tau \ria^2\lea \xi, \zeta \ria^2.
\]
Using Cauchy-Schwarz, it is sufficient to verify that
\[\lea\Rc \s, \s \ria\lea\Rc \tau, \tau \ria\geq\lea\Rc \s, \tau \ria^2.
\]
And this holds since $\Rc$ is positive semi-definite (resp. definite) on $S$.
\endproof
\bc For $E$ flat, the Weitzenb\"ock operator of $S\ot E$ is positive semi-def\-i\-nite (resp. definite) if and only if the Weitzenb\"ock operator of $S$ is positive semi-def\-i\-nite (resp. definite). 
\ec
\proof This is immediate from the fact that $\Theta$ vanishes identically in this case. 
\endproof


\section{Twistor sections}
The purpose of this section is to prove Theorem B and Theorem C.  To do this, we first introduce the twistor and Killing equations for sections of Dirac bundles and we establish some natural extensions of well-known facts in Spin Geometry.

\bd[{\bf Twistor section}\hspace{-4pt} ]\label{twistordef} A section $\s$ of a Dirac bundle $S$ is a twistor section if it satisfies 
 \begin{equation}\label{twistorspi}
\nb_X\s + \frac{1}{n}X\cdot D\s =0,
\end{equation}
for all $X\in TM$.  
\ed
\bd[{\bf Killing section}\hspace{-4pt} ]\label{twistordef} A section $\s$ of a Dirac bundle $S$ is a Killing section with constant $\la$ if it satisfies 
\begin{equation}\label{Killingspi}
\nb_X\s = \la X\cdot \s
\end{equation}
for all $X\in TM$.  
\ed
Equations \eqref{twistorspi} and \eqref{Killingspi} are respectively the twistor and Killing equations on $S$. 

The following result is a natural extension to Dirac bundles of a result for spin bundles (see the proposition in \cite{Friedrich} p. 117). 
\bp Let $\s$ be a Killing section with constant $\la$ of a Dirac bundle $S$ over an $n$-dimensional $M$. 
\begin{enumerate}[\indent\rm(i)]
\item If $M$ is connected, then $\s$ vanishes identically if it vanishes at a point. 
\item The section $\s$ is also twistor. Moreover, it is an eigensection of the Dirac operator with eigenvalue $-n\lambda$.
\end{enumerate}
\ep
\proof
The proof is straightforward and it is essentially identical to the proof of the proposition in \cite{Friedrich} p. 117. The difference is that therein a hermitian inner product is used.\footnote{A result corresponding to the third part of that proposition is not included here  since in the euclidean case $\lea e_i\cdot \s,\s\ria=0$, and thus the vector field  $\sum\lea e_i\cdot \s,\s\ria e_i$  vanishes identically.}
\endproof
In the case of spin bundles, the classical Bochner identity implies explicit conditions for the twistor equation to be satisfied. These formulas relate the Dirac operator and the curvature to a solution of the twistor equation. In the Dirac bundle case, the following result shows that there also exist explicit conditions linking the generalized Bochner identity \eqref{GenBochnerId} and the twistor equation \eqref{twistorspi}. It is a generalization of Proposition 2 in \cite{Lich2} p. 335 and of the first two formulas of Theorem 3 in \cite{BFGK} p. 24. In analogy to the expressions found in the latter, we define the Ricci curvature operator $\Ric_X:S\to S$  of a Dirac bundle $S$  by 
$$\Ric_X\s = 2\sum_{i=1}^{n}e_i\pc R(e_i,X)\s,$$
where $e_1,\ldots,e_n$ is a local orthonormal frame. 

\begin{theo}\label{twistorchar}
Let $S$ be a Dirac bundle over a compact $M$. A section $\s$ of $S$ is a twistor section if and only if
\begin{equation}\label{twistorequiv}
D^2\s = \frac{n}{n-1}\Rc\s.
\end{equation}
Moreover, in this case,
\begin{equation}\label{eqtRic}
\nb_XD\s = \frac{n}{n-2}\left[\frac{1}{n-1}X\pc \Rc-\frac{1}{2}\Ric_X\right]\s.
\end{equation}
\end{theo}
\begin{proof}
Let $\s$ be a twistor section of $S$ and let $e_1,\ldots, e_n$ be a local orthonormal frame.  Then from \eqref{defconlap} and\eqref{twistorspi},
\begin{align*}
\nb^{*}\nb\s &= -\sum_{j=1}^n\nb^2_{e_j,e_j}\s =-\sum_{j=1}^{n}\left(\nb_{e_j}\nb_{e_j}\s-\nb_{\nb_{e_j}e_j}\s \right)\\
&= \frac{1}{n}\sum_{j=1}^{n}\left(\nb_{e_j}(e_j\pc D\s)- \nb_{e_j}e_j\pc D\s\right)=\frac{1}{n}D^2\s.
\end{align*}
At this point $\eqref{twistorequiv}$ follows from the general Bochner identity \eqref{GenBochnerId}. 

Conversely, if $\s\in\Ga(S)$ satisfies \eqref{twistorequiv} then,  again using \eqref{GenBochnerId}, $\s$ satisfies 
\[
\nb^{*}\nb\s =\frac{1}{n}D^2\s.
\] 
Using this, \eqref{Dselfad}, \eqref{CLapSelfad} and \eqref{innprodnabla}, it follows that

\begin{align*}
\sum_{j=1}^{n}\|\nb_{e_j}\s+\frac{1}{n}e_j\pc D\s\|^2
&= \sum_{j=1}^{n}\Big\{\|\nb_{e_j}\s\|^2+\frac{2}{n}(e_j\pc D\s, \nb_{e_j}\s)+\frac{1}{n}\|D\s\|^2\Big\}\\
&= (\nb^{*}\nb\s, \s) - \frac{2}{n}(D\s, D\s)+\frac{1}{n}(D\s, D\s)\\
&= (\nb^{*}\nb\s, \s) - \frac{1}{n}(D^2\s,\s)=0.
\end{align*}
So that $\nb_{e_j}\s+\frac{1}{n}e_j\pc D\s =0$ for $j=1,\ldots, n$; from which $\eqref{twistorspi}$ follows.

To prove \eqref{eqtRic} let $\s$ be a twistor section, and $X, Y\in TM$ over the same point.  Then,
\be\label{Riemtwist}
R(X,Y)\s=\frac1n\big(X\cdot\nb_YD\s-Y\cdot \nb_XD\s\big).
\ee
Indeed, for  two commuting vector fields $X,Y\in\X(M)$,  
\begin{align*}
R(X,Y)\s&=\nb_X\nb_Y\s- \nb_Y\nb_X\s\\
&=\frac1n\big(\nb_Y(X\cdot D\s ) -\nb_X(Y\cdot D\s )\big)\\
&=\frac1n\big(X\cdot\nb_YD\s-Y\cdot \nb_XD\s\big).
\end{align*}
Since this is tensorial \eqref{Riemtwist} holds. Lastly, using \eqref{twistorequiv} and \eqref{Riemtwist},
\begin{align*}
\frac12\Ric_X\s&=\sum_{i=1}^n e_i\cdot R(e_i,X)\s\\
&=\frac1n\sum_{i=1}^n\big(e_i\cdot e_i\cdot\nb_XD\s-e_i\cdot X\cdot \nb_{e_i}D\s\big)\\
&=-\nb_XD\s+\frac1nX\cdot D^2\s+\frac{2}{n}\sum_{i=1}^n \lea e_i,X\ria\nb_{e_i}D\s\\
&=-\nb_XD\s+\frac{1}{n-1}X\cdot\Rc\s+\frac2n\nb_XD\s\\
&=\frac{2-n}{n}\nb_XD\s+\frac{1}{n-1}X\cdot\Rc\s,
\end{align*}
from which \eqref{eqtRic} follows.
\end{proof}
Consider for each $X\in\X(M)$, the endomorphism $\K_X:S\to S$ given by
\be
\K_X=\frac{n}{n-2}\left[\frac{1}{n-1}L_X\circ\Rc-\frac{1}{2}\Ric_X\right],
\ee
where $L_X:S\to S$ is the Clifford multiplication by $X$. By \eqref{twistorspi} and \eqref{eqtRic}, 

\be\label{paratwist}
\nb_X\s=-\frac1nL_X\circ D\s\hspace{2cm} \nb_X D\s= \K_X\s.
\ee
which suggests the following result (cf. \cite{BFGK} p. 26).
\bt\label{parallel}
On $E=S\oplus S$, with respect to the connection $\nb^E:\Ga(E)\to\Ga(T^*M\otimes E)$ given by the matrix
\[
\begin{pmatrix} \nb& \frac1nL\\
-\K &\nb
\end{pmatrix}
\]
a section $\s$ is a twistor section if and only if $(\s, D\s)$ is $\nb^E$-parallel. 
\et
\proof
That for a twistor section $\s$, $\nb^E(\s,D\s)=0$ is immediate from \eqref{paratwist}. Conversely, let $(\s,\tau)$ be $\nb^E$-parallel. Then $\nb_X\s=-\frac1n X\cdot\tau$ holds true for all $X\in TM$ and thus
\[
D\s=\sum_{i=1}^ne_i\cdot \nb_{e_i}\s=\sum_{i=1}^ne_i\cdot (-\frac1n e_i\cdot \tau)=\tau,
\]
which establishes the claim.
\endproof

\bc
A twistor section $\s$ is determined by $\s_{p_0}$ and $D\s_{p_0}$ for a single $p_0$.
\ec
\bc 
The space of solutions to the twistor equation is finite-dimensional and its dimension is bounded by $2\,\r S$.
\ec
\bc  If for a twistor section $\s$, both $\s$ and $D\s$ vanish simultaneously at a point, then $\s$ vanishes identically.
\ec
The form of \eqref{Killingspi}, as well as Theorem \ref{parallel}, suggest the use of modified connections (cf. \cite{Friedrich} pp. 114-115) such as the one in the following technical result. 
\bl Let $f:M\to \R$ be any smooth function and let $\nb^f$ be the covariant derivative on S given by
\be
\nb^{f}_X\s = \nb_X\s + fX\cdot \s.
\ee
Then $\nb^f$ is metric,
\be\label{nbflaplacian}
\nb^{f*}\nb^f \s = \nb^{*}\nb\s - \nb f\cdot \s -2fD\s +nf^2\s,
\ee
and
\be\label{eqDf}
(D-f)^2\s =\nb^{f*}\nb^f\s + \Rc\s + (1-n)f^2\s.
\ee
\el
\proof To see that $\nb^f$ is metric, notice that
\begin{align*}
\lea\nb_X^f\s,\tau \ria+ \lea\s,\nb_X^f\tau \ria&=\lea\nb_X\s,\tau \ria+ \lea\s,\nb_X\tau \ria.
\end{align*}
Now, let $e_1,\ldots,e_n$ be a local orthonormal frame.  Then
\begin{align*} 
\nb^f_{e_j}\nb^f_{e_j}\s&=\nb_{e_j}\nb_{e_j} \s + \nb_{e_j}( fe_j\cdot\s)+fe_j\cdot\nb_{e_j}\s+ f e_j\cdot(fe_j\s)\\
 &=\nb_{e_j}\nb_{e_j} \s +e_j(f)e_j\cdot \s + f\nb_{e_j}e_j\cdot\s +2fe_j\cdot\nb_{e_j}\s -f^2\s,
 \end{align*}
and
\begin{align*} 
\nb^f_{\nb_{ e_j}e_j}\s &=\nb_{\nb_{ e_j}e_j}\s+f\nb_{ e_j}e_j\cdot\s.
\end{align*}
From these two expressions, it follows that
\begin{align*} 
\nb^{f*}\nb^f\s &=-\sum_{j=1}^{n}[\nb^{f}]^2_{e_j,e_j}\s\\
 & =-\sum_{j=1}^{n}( \nb^f_{e_j}\nb^f_{e_j}\s-\nb^f_{\nb_{ e_j}e_j}\s) \\
 &= -\sum_{j=1}^{n} (\nb^2_{e_j,e_j} \s +e_j(f)e_j\cdot \s +2fe_j\cdot\nb_{e_j}\s -f^2\s )\\
 &= \nb^{*}\nb\s - \nb f\cdot \s -2fD\s +nf^2\s,
\end{align*}
which establishes \eqref{nbflaplacian}. To see \eqref{eqDf}, use \eqref{LeibD} and the general Bochner identity, 
\begin{align*}
(D-f)^2\s &= D^2\s -2fD\s - \nb f\cdot\s +f^2\s \\
&=\nb^{*}\nb\s+\Rc\s-2fD\s-\nb f\cdot\s+f^2\s \\
\pushQED{\qed} 
&= \nb^{f*}\nb^f\s + \Rc\s + (1-n)f^2\s.\qedhere
\popQED
\end{align*}
The following result characterizes Killing sections on Dirac bundles over compact manifolds.
\begin{theo}\label{equiv1}
On a Dirac bundle $S$ over a compact $n$-dimensional $M$ the following are equivalent.
\begin{enumerate}[\indent\rm(i)]
\item A section $\s$ of $S$ is a Killing section with constant $\la$.
\item A section $\s$ of $S$ satisfies $D\s=-n\la\s$ and $\Rc\s = \la^2n(n-1)\s$.
\end{enumerate}
\end{theo}
\begin{proof}
Let $\s$ be as in (i), then 
\[
D\,\s = \sum_{i=1}^{n}e_i\cdot \nb_{e_i} \s = \la \sum_{i=1}^{n}e_i\cdot e_i\cdot \s = -n\la\, \s.
\]
Thus $D^2\s=n^2\la^2\s$, which together with  \eqref{twistorequiv} establishes the remaining formula in (ii).

Conversely, let $\s$ be as in (ii) and use  \eqref{eqDf} with $f=-\la$ to obtain 
\[
(D+\la)^2\s = \nb^{-\la*}\nb^{-\la}\s + \Rc\s + (1-n)\la^2\s.
\]
Using both conditions of (ii) in this last formula immediately yields that
\[
\nb^{-\la*}\nb^{-\la}\s=0.
\]
Recall that the notion of connection laplacian can be defined on euclidean vector bundles and that properties such as \eqref{CLapSelfad} still hold. In particular, this is the case of $S$ with metric connection $\nb^{-\la}$.  Therefore, since $M$ is compact it follows that $\s$ is $\nb^{-\la}$-parallel.  This is exactly the Killing condition \eqref{Killingspi} for constant $\la$.
\end{proof}
To end this section we provide another characterization of the Killing condition in terms of a lower bound for the square of the eigenvalues of the Dirac operator. 
\begin{prop}
Let $S$ be  a Dirac bundle over a compact $M$ and let $\mu$ be an eigenvalue of $D$, then
\be\label{Dlambda}
\mu^2 \geq \frac{n}{n-1}R_0,
\ee
where 
\[
R_0=\min_{|\s|=1}\lea\Rc\s,\s \ria.
\]
Moreover, the equality holds if and only if $\s$ is a Killing section with constant $\mp \sqrt{\frac{R_0}{n(n-1)}}$ and $\Rc$ is positive semi-definite.
\end{prop}

\begin{proof}
Let $\s$ be an eigensection of $D$ with eigenvalue $\mu$.  On \eqref{eqDf} let $f=\frac{\mu}{n}$ to obtain
\[
\mu^2\frac{n-1}{n}\s = \nb^{\frac{\mu}{n}*}\nb^{\frac{\mu}{n}}\s + \Rc\s.
\]
Since $M$ is compact, \eqref{CLapSelfad} yields
\[
\mu^2\frac{n-1}{n}(\s, \s) = (\nb^{\frac{\mu}{n}}\s, \nb^{\frac{\mu}{n}}\s) + (\Rc\s,\s)\geq R_0(\s,\s),
\]
which implies \eqref{Dlambda}.

Furthermore, if equality holds, then $\s$ is $\nb^{\frac{\mu}{n}}$-parallel; i.e. $\s$ is Killing with constant $\lambda=-\frac{\mu}{n}=\mp  \sqrt{\frac{R_0}{n(n-1)}} $. Conversely, if $\s$ is a Killing section with constant $\lambda=\mp  \sqrt{\frac{R_0}{n(n-1)}} $, then by Theorem \ref{equiv1},  
\[
D\s=-n\lambda\s=\pm\sqrt{\frac{n}{n-1}R_0}\,\s.
\]
Hence, $\mu=\pm\sqrt{\frac{n}{n-1}R_0}$ is an eigenvalue of $D$ and equality in \eqref{Dlambda} holds. 
\end{proof}

%

\section{The Clifford bundle case}
The main purpose of this section is to prove Theorem D.  First, we study the interplay of the Weitzenb\"ock operator with the additional structure of the Clifford bundle. Then, in Proposition \ref{Trace} we prove that the trace of the Weitzenb\"ock operator when restricted to $p$-forms is a multiple of the scalar curvature.  Lastly, we restrict our attention to spaces of constant sectional curvature and obtain explicit formulas for the Weitzenb\"ock operator.

Under the standard identification with the exterior bundle, the Dirac operator is given as  $D=d+d^*$, from which
 \be\label{D2Cliff}
 D^2=\Delta,
 \ee
where $\Delta$ is the usual Hodge laplacian (see \cite{Lawson} p. 123).  A simple computation shows that for any  function $f$ and any $n$-form $\omega$
\be\label{Rgari0nforms}
\Rc(f)=\Rc(\omega)=0.
\ee
Furthermore, for any $1$-form $\varphi$, 
 \be\label{RcRic}
 \Rc\varphi=\Ric(\varphi)
 \ee
 and $\Rc$ is positive semi-definite (resp. definite) whenever the curvature operator $R:\Lambda^2(M)\to \Lambda^2(M)$ is positive semi-definite (resp. definite) (see \cite{Lawson} pp. 156-160).  

Following the standard multi-index notation for forms (cf.  \cite{Kobayashi}), for any $e_1,\ldots,e_n$ local orthonormal frame and for any ordered $I\subset \{1,\ldots, n\}$ denote by $e_I$ the corresponding exterior form and denote by $I'= \{1,\ldots, n\}\setminus I$, its ordered complement.  Since the Hodge operator $\ast$ is given by the equation 
\be
\vp\wedge\ast\psi =\lea\vp,\psi\ria \vol,
\ee
where $\vol=e_{\{1,\ldots,n\}}$, it follows $\ast e_I=(-1)^{\ell_I}e_{I'}$, for an appropriate $\ell_I$.
\bp On the Clifford bundle the Weitzenb\"ock operator and the Hodge operator commute.
\ep
\begin{proof} 

In view of $\eqref{GenBochnerId}$, it is enough to prove that $\ast$ commutes with $D^2$ and $\nb^{*}\nb$.  Since the operator $D^2$ satisfies \eqref{D2Cliff}, it commutes with $\ast$. Hence, by \eqref{defconlap}, it is enough to prove that $\nb$ commutes with $\ast$. To see this, notice that 
\be\label{volformula}
\vol\cdot \varphi=(-1)^{p(n-p)+\frac{p(p+1)}{2}}\ast\varphi
\ee
 for any $p$-form $\varphi$ (see (5.35) in \cite{Lawson} p. 129). Now, since $\vol$ is parallel,
\begin{align*}
\nb_X(\ast \varphi)&=(-1)^{p(n-p)+\frac{p(p+1)}{2}}\nb_X(\vol\cdot\varphi)\\
&=(-1)^{p(n-p)+\frac{p(p+1)}{2}}\vol\cdot\nb_X\varphi\\&=\ast\nb_X\varphi,
\end{align*}
where \eqref{volformula} is used for $\varphi$ and $\nb_X\varphi$. 
\end{proof}
\bp\label{Trace} On the exterior bundle of $M$, for any $p=1,\ldots,n-1$, the trace of the restriction $\Rc_p$ of the Weitzenb\"ock operator to $p$-forms is
\be\label{trRcp}
\tr\,\Rc_p= {{n-2}\choose{p-1}}s,
\ee
where $s$ is the scalar curvature of $M$. In particular, in light of \eqref{Rgari0nforms}, the trace of the total Weitzenb\"ock operator is
\be\label{trR}
\tr\,\Rc=2^{n-2}s.
\ee
\ep
\begin{proof} 
As is well known, the curvature operator on  $2$-forms is determined by
\[
\lea {\bf R} (e_i\wedge e_j), e_k\wedge e_\ell\ria = \lea R(e_i,e_j)e_\ell,e_k\ria,
\]
and it is self-adjoint.   In terms of this operator, 
\[
\lea\Rc \vp, \vp \ria= \frac14\sum_{\xi,\zeta\in B} \lea{\bf R} \xi, \zeta\ria\lea [\xi,\vp], [\zeta,\vp] \ria,
\]
for any orthonormal basis $B$ of the space of $2$-forms (cf. \cite{Lawson} p. 159).   Using the above formula, the trace of $\Rc_p$ is given by
\be\label{trRcp1}
\tr\,\Rc_p=\sum_{|I|=p}\lea \Rc e_I, e_I\ria=\frac14\sum_{|I|=p}\sum_{\xi\in B} \lambda_\xi \big|[\xi,e_I]\big|^2,
\ee
assuming $B$ is a basis of eigen-$2$-forms of ${\bf R}$, with $\lambda_\xi$ the corresponding eigenvalue for $\xi\in B$.  To compute the right-hand side of \eqref{trRcp1}, let $\xi=\frac12\sum\xi^{ij}e_ie_j$ to yield that
\be\label{adjointnorm}
 \big|[\xi,e_I]\big|^2=\frac14\sum_{i,j,k,\ell}\xi^{ij}\xi^{k\ell}\lea [e_ie_j,e_I],[e_ke_\ell,e_I] \ria.
\ee
An easy computation gives that 
\be\label{Lawsonadjoint}
[e_ie_j,e_I]=\begin{cases}0 &{\rm if }\; \{i,j\}\subset I \\ 0 &{\rm if }\; \{i,j\}\subset I'\\ 2e_ie_je_I &{ \rm otherwise,}\end{cases}
\ee
from which it follows that $\lea [e_ie_j,e_I],[e_ke_\ell,e_I] \ria$ is nonzero only when $\{i,j\}=\{k,\ell\}$ provided that $\{i,j\}$ is not contained in either of $I$ or $I'$.  Taking into account these four cases, \eqref{adjointnorm} and \eqref{Lawsonadjoint} imply that

\begin{align*}
 \big|[\xi,e_I]\big|^2&= \sum_{i\in I}\sum_{j\in I'} (\xi^{ij})^2-\sum_{i\in I}\sum_{j\in I'} \xi^{ij}\xi^{ji}-\sum_{i\in I'}\sum_{j\in I} \xi^{ij}\xi^{ji}+\sum_{i\in I'}\sum_{j\in I} (\xi^{ij})^2\\
 &=4 \sum_{i\in I}\sum_{j\in I'} (\xi^{ij})^2,
\end{align*}
since $\xi^{ji}=-\xi^{ij}$. Replace $ \big|[\xi,e_I]\big|^2$ by this last expression in \eqref{trRcp1} to obtain
\be\label{trRcp2}
\tr\, \Rc_p=\sum_{\xi\in B} \lambda_\xi \sum_{|I|=p}\sum_{i\in I}\sum_{j\in I'} (\xi^{ij})^2.
\ee
Observe that for fixed $i$ and $j$, there are $n-2\choose p-1$ subsets $I$ for which $i\in I$ and $j\in I'$.  This implies that

\be\label{normxi}
 \sum_{|I|=p}\sum_{i\in I}\sum_{j\in I'} (\xi^{ij})^2={n-2\choose p-1}\sum_{i,j}(\xi^{ij})^2=2{n-2\choose p-1},
\ee
since $\xi\in B$ has norm 1. Finally, since
\[
\sum_{\xi\in B} \lambda_\xi=\tr\,{\bf R}=\frac s2,
\]
use  \eqref{trRcp2} and \eqref{normxi} to yield \eqref{trRcp}.  A standard computation yields \eqref{trR}. 
\end{proof}
Theorem \ref{twistorchar} can be used to solve the twistor equation completely on $\ClS$.  In order to do that the following preliminary result is needed, which might be of independent interest. 
\bl\label{lemtwsp} Let $M$ have constant sectional curvature $\kappa$. On the exterior bundle, the restriction $\Rc_p$ of the Weitzenb\"ock operator to $p$-forms is
\be
\Rc_p\varphi=\kappa p(n-p)\varphi.
\ee
\el
\proof
Let $e_1,\ldots, e_n$ be a local orthonormal frame.   From \cite{Lawson} p. 111,  the curvature operator on $\Cl$  can be written as
\be
R(X,Y)\varphi = \frac{1}{2}\sum_{k < l}\lea R(X,Y) e_k, e_l\ria [e_k\pc e_l,\varphi].
\ee
In particular, it follows from this formula and \eqref{Lawsonadjoint} that
\begin{align*}
R(e_i, e_j)e_I&= \frac{1}{4}\sum_{k, l}\lea R(e_i,e_j) e_k, e_l\ria [e_k\pc e_l,\varphi]  \equiv\frac14\sum_{k,\ell}R_{ijk\ell} [e_k\pc e_l,\varphi]\\
&=\frac12\sum_{k\in I}\sum_{\ell\in I'}R_{ijk\ell}\,e_k\cdot e_\ell\cdot  e_I+\frac12\sum_{k\in I'}\sum_{\ell\in I}R_{ijk\ell}\,e_k\cdot e_\ell \cdot e_I\\
&=\sum_{k\in I}\sum_{\ell\in I'}R_{ijk\ell}\,e_k\cdot e_\ell\cdot  e_I.
\end{align*}
Since $\kappa$ is constant,  $R(X,Y)Z=\kappa(\lea Y, Z \ria X- \lea X, Z \ria Y)$, and thus
\begin{align*}
\Rc_pe_I&=\frac12 \sum_{i,j}\sum_{k\in I}\sum_{\ell\in I'} \kappa(\delta_{i\ell}\delta_{jk}-\delta_{ik}\delta_{j\ell})e_i\cdot e_j\cdot e_k\cdot e_\ell \cdot e_I\\
\pushQED{\qed} 
&= \sum_{k\in I}\sum_{\ell\in I'}\kappa e_I=\kappa p (n-p) e_I.\qedhere
\popQED
\end{align*}
\begin{theo}\label{cliffSntwistor} 
The twistor sections of $\ClS$, $n\geq 2$, are given by
\be\label{twistCliff}
\s = c_1+df_1+d^*(\ast f_2) + \ast c_2,
\ee
with constants $c_1, c_2$ and functions $f_1, f_2: S^n\to\R$ such that $\Delta f_i =n f_i$.
\end{theo}
Solutions to $\Delta f =nf$ are well known to exist (see \cite{MR2002701} for details).
\begin{proof}[Proof of Theorem \ref{cliffSntwistor}]
From Theorem $\ref{twistorchar}$ to find twistor sections is equivalent to solving \eqref{twistorequiv}. Let $\s=\s_0+\cdots+\s_n$, where $\s_p$ is of degree $p$.  From \eqref{D2Cliff}  and Lemma \ref{lemtwsp}, $\s$ is a twistor section if and only if
\begin{equation}\label{teqsp}
\Delta \s_p= \frac{n}{n-1}p(n-p)\s_p,\qquad p=0, 1,\,\ldots, n.
\end{equation}
Hereinafter the fact that $\Delta$ commutes with $d$ and $d^*$ is used extensively; e.g. if $\s_0=g$ and $\s_n=\ast h$, by \eqref{teqsp} $g$ and $h$ are harmonic and thus constant.  

Now, for $p\neq0$ and $p\neq n$, the smallest positive eigenvalue $\mu_{n,p}$ for $\Delta$ on the space of $p$-forms of $S^n$ is given as  the minimum 
 \be\label{Deltavalue}
\mu_{n,p}=\min\{p(n-p+1), (p+1)(n-p)\},
\ee
where $p(n-p+1)$ (resp. $(p+1)(n-p)$) corresponds to the first  positive eigenvalue when restricted to closed  (resp. co-closed) $p$-forms.\footnote{These formulas are due to Calabi (unpublished). See \cite{GueSa} for details.}

For $p=1$, \eqref{teqsp} is written as $\Delta\s_1=n\s_1$ and thus 
\be\label{Deltas1}
\s_1=df_1+d^*\eta
\ee
with $f_1,\eta\in\ker(\Delta-n)$. The case $n=2$ follows since $\eta=\ast f_2$.  For $n=3$ and $p=1$, since $p(n-p+1)=3$, $(p+1)(n-p)=4$ and \eqref{teqsp} is $\Delta\s_1=3\s_1$, it follows that any solution $\s_1$ must be  closed and thus in \eqref{Deltas1} $d^*\eta=0$. For $n=3$ and $p=2$ the analysis is analogous. Therefore, there exist functions $f_1$ and $f_2$ in the kernel of $\Delta-n$ such that 
\be\label{Eigen1n-1}
\s_1=df_1\quad\text{ and }\quad \s_{n-1}=d^*(\ast f_2),
\ee
which proves the claim for $n=3$. For $n\geq 4$, notice that for $p=2,\ldots, n-2$, 
\[
\frac{n}{n-1}p(n-p)<\mu_{n,p},
\]
and thus the only solution $\s_p$ to \eqref{teqsp} is the trivial solution; therefore
\[\s=\s_0+\s_1+\s_{n-1}+\s_n.
\]
Since $n<\mu_{n,2}$, for $p=1$ it follows that in  \eqref{Deltas1} $\eta\equiv 0$. The case $p=n-1$ is analogous. Therefore, there exist functions $f_1$ and $f_2$ as in \eqref{Eigen1n-1}, which finishes the proof.
\end{proof}


\phantomsection
\addcontentsline{toc}{section}{A. The two-dimensional case}
\renewcommand\thesection{\Alph{section}}
\setcounter{section}{0}
\tocless\section{The two-dimensional case}
Here we analyze the Killing equation on Clifford bundles over surfaces. This equation is much more restrictive than the twistor equation as the following result shows. \\

\bt\label{Kill2D} For a compact $2$-dimensional $M$, there are nontrivial Killing sections of $\Cl$  with value $\lambda$ if and only if $\lambda=0$ and  the gaussian curvature $\kappa\equiv 0$.  Furthermore any such section is parallel.
\et
To prove this, a couple of preliminary observations are required.   Firstly,  using the general fact that $\Cl$ splits into a Whitney sum of its even and odd parts, it is easy to see that the Clifford multiplication and the Dirac operator map the even into the odd and viceversa, while the Weitzenb\"ock operator preserves them. 

By \eqref{Rgari0nforms} and \eqref{RcRic}, it follows that for a two-dimensional $M$,
\be\label{curv2}
\Rc\s=\kappa\s_1,
\ee
for any section $\s\in\Cl$,  where $\s_1$ is the odd part of $\s$.
\bl\label{eigenD2Dlem} For a compact $2$-dimensional $M$, let $\s$ an eigensection of $D$ with eigenvalue $\lambda\neq0$.  Then the even and odd parts of $\s$ are non trivial and 
\be\label{eigenD2D}
\lambda^2\geq \min\kappa.
\ee 
\el
\proof   Let $\s_0$ and $\s_1$ be the even and odd parts of $\s$, respectively. Since $D$ shifts degree, $D\s_0 = \la \s_1$ and $D\s_1 = \la \s_0$. In particular, since $\lambda\neq0$ and $\s$ is nontrivial then so are $\s_0$ and $\s_1$. Finally,  using the Bochner identity and \eqref{curv2},
\[
\la^2 \|\s_1\|^2 = \|\nb\s_1\|^2 + \ka\|\s_1\|^2\geq \ka\|\s_1\|^2,
\]
from which \eqref{eigenD2D} follows.
\endproof
The estimate from \eqref{eigenD2D} is slightly better than that of \eqref{Dlambda}, albeit in this trivial case.
\proof[Proof of Theorem \ref{Kill2D}]
Since $\s$ is a nontrivial Killing section, Theorem $\ref{equiv1}$ implies that 
\be\label{TeoC2D}
\Rc\s=2\la^2\s\quad\text{and}\quad D\s=-2\la \s.
\ee
 In light of \eqref{RcRic}, the first equation is $\ka\s_1=2 \la^2\s$, which is satisfied if and only if $2\la^2\s_0 =0$ and $\ka = 2\la^2$.  This, together with the second equation in \eqref{TeoC2D} and Lemma \ref{eigenD2Dlem} implies that $\lambda=0$ and $\kappa\equiv 0$.  Conversely, $\lambda=0$ means that $\s$ is parallel; $\kappa\equiv 0$ that $M$ is flat, which guarantees the existence of nontrivial parallel sections. 
\endproof

{\bf Acknowledgements.}  The first two authors are supported by the  C\'a\-te\-dras Conacyt Program Project No. 61. The third  author was supported by the  FORDECYT Grant No. 265667 (CONACYT) as a postdoctoral fellow.   The third author wishes to express his gratitude to the Instituto de Ma\-te\-m\'a\-ti\-cas Oaxaca (UNAM) for its hospitality, where most of this project was developed.


\bibliographystyle{plainnat}
\bibliography{../referencias}

\begin{thebibliography}{20}
\providecommand{\natexlab}[1]{#1}
\providecommand{\url}[1]{\texttt{#1}}
\expandafter\ifx\csname urlstyle\endcsname\relax
  \providecommand{\doi}[1]{doi: #1}\else
  \providecommand{\doi}{doi: \begingroup \urlstyle{rm}\Url}\fi

\bibitem[Atiyah and Ward(1977)]{AW}
M.~F. Atiyah and R.~S. Ward.
\newblock Instantons and algebraic geometry.
\newblock \emph{Comm. Math. Phys.}, 55\penalty0 (2):\penalty0 117--124, 1977.
\newblock ISSN 0010-3616.
\newblock URL \url{http://projecteuclid.org/euclid.cmp/1103900980}.

\bibitem[Atiyah et~al.(1978)Atiyah, Hitchin, and Singer]{AHS}
M.~F. Atiyah, N.~J. Hitchin, and I.~M. Singer.
\newblock Self-duality in four-dimensional {R}iemannian geometry.
\newblock \emph{Proc. Roy. Soc. London Ser. A}, 362\penalty0 (1711):\penalty0
  425--461, 1978.
\newblock ISSN 0962-8444.
\newblock \doi{10.1098/rspa.1978.0143}.
\newblock URL \url{https://doi.org/10.1098/rspa.1978.0143}.

\bibitem[B\"{a}r(1996)]{Bar}
Christian B\"{a}r.
\newblock The {D}irac operator on space forms of positive curvature.
\newblock \emph{J. Math. Soc. Japan}, 48\penalty0 (1):\penalty0 69--83, 1996.
\newblock ISSN 0025-5645.
\newblock \doi{10.2969/jmsj/04810069}.
\newblock URL \url{https://doi.org/10.2969/jmsj/04810069}.

\bibitem[Baum et~al.(1991)Baum, Friedrich, Grunewald, and Kath]{BFGK}
Helga Baum, Thomas Friedrich, Ralf Grunewald, and Ines Kath.
\newblock \emph{Twistors and {K}illing spinors on {R}iemannian manifolds},
  volume 124 of \emph{Teubner-Texte zur Mathematik [Teubner Texts in
  Mathematics]}.
\newblock B. G. Teubner Verlagsgesellschaft mbH, Stuttgart, 1991.
\newblock ISBN 3-8154-2014-8.
\newblock With German, French and Russian summaries.

\bibitem[Berger(2003)]{MR2002701}
Marcel Berger.
\newblock \emph{A panoramic view of {R}iemannian geometry}.
\newblock Springer-Verlag, Berlin, 2003.
\newblock ISBN 3-540-65317-1.
\newblock \doi{10.1007/978-3-642-18245-7}.
\newblock URL \url{https://doi.org/10.1007/978-3-642-18245-7}.

\bibitem[Friedrich(1990)]{Friedrich1}
Thomas Friedrich.
\newblock On the conformal relation between twistors and {K}illing spinors.
\newblock In \emph{Proceedings of the {W}inter {S}chool on {G}eometry and
  {P}hysics ({S}rn\'{\i}, 1989)}, number~22, pages 59--75, 1990.

\bibitem[Friedrich(2000)]{Friedrich}
Thomas Friedrich.
\newblock \emph{Dirac operators in {R}iemannian geometry}, volume~25 of
  \emph{Graduate Studies in Mathematics}.
\newblock American Mathematical Society, Providence, RI, 2000.
\newblock ISBN 0-8218-2055-9.
\newblock \doi{10.1090/gsm/025}.
\newblock URL \url{https://doi.org/10.1090/gsm/025}.
\newblock Translated from the 1997 German original by Andreas Nestke.

\bibitem[Friedrich and Pokorn\'{a}(1991)]{FriedPok}
Thomas Friedrich and Olga Pokorn\'{a}.
\newblock Twistor spinors and solutions of the equation ({E}) on {R}iemannian
  manifolds.
\newblock In \emph{Proceedings of the {W}inter {S}chool on {G}eometry and
  {P}hysics ({S}rn\'{\i}, 1990)}, number~26, pages 149--154, 1991.

\bibitem[Guerini and Savo(2004)]{GueSa}
Pierre Guerini and Alessandro Savo.
\newblock Eigenvalue and gap estimates for the {L}aplacian acting on
  {$p$}-forms.
\newblock \emph{Trans. Amer. Math. Soc.}, 356\penalty0 (1):\penalty0 319--344,
  2004.
\newblock ISSN 0002-9947.
\newblock \doi{10.1090/S0002-9947-03-03336-1}.
\newblock URL \url{https://doi.org/10.1090/S0002-9947-03-03336-1}.

\bibitem[Habermann(1990)]{Habermann}
Katharina Habermann.
\newblock The twistor equation on {R}iemannian manifolds.
\newblock \emph{J. Geom. Phys.}, 7\penalty0 (4):\penalty0 469--488 (1991),
  1990.
\newblock ISSN 0393-0440.
\newblock \doi{10.1016/0393-0440(90)90002-K}.
\newblock URL \url{https://doi.org/10.1016/0393-0440(90)90002-K}.

\bibitem[Habermann(1994)]{Habermann1}
Katharina Habermann.
\newblock Twistor spinors and their zeroes.
\newblock \emph{J. Geom. Phys.}, 14\penalty0 (1):\penalty0 1--24, 1994.
\newblock ISSN 0393-0440.
\newblock \doi{10.1016/0393-0440(94)90051-5}.
\newblock URL \url{https://doi.org/10.1016/0393-0440(94)90051-5}.

\bibitem[Kobayashi(1987)]{Kobayashi}
Shoshichi Kobayashi.
\newblock \emph{Differential geometry of complex vector bundles}, volume~15 of
  \emph{Publications of the Mathematical Society of Japan}.
\newblock Princeton University Press, Princeton, NJ; Princeton University
  Press, Princeton, NJ, 1987.
\newblock ISBN 0-691-08467-X.
\newblock \doi{10.1515/9781400858682}.
\newblock URL \url{https://doi.org/10.1515/9781400858682}.
\newblock Kan\^{o} Memorial Lectures, 5.

\bibitem[Kobayashi and Nomizu(1996)]{KNvol1}
Shoshichi Kobayashi and Katsumi Nomizu.
\newblock \emph{Foundations of differential geometry. {V}ol. {I}}.
\newblock Wiley Classics Library. John Wiley \& Sons, Inc., New York, 1996.
\newblock ISBN 0-471-15733-3.
\newblock Reprint of the 1963 original, A Wiley-Interscience Publication.

\bibitem[K\"{u}hnel and Rademacher(1994)]{KuRa}
Wolfgang K\"{u}hnel and Hans-Bert Rademacher.
\newblock Twistor spinors with zeros.
\newblock \emph{Internat. J. Math.}, 5\penalty0 (6):\penalty0 877--895, 1994.
\newblock ISSN 0129-167X.
\newblock \doi{10.1142/S0129167X94000450}.
\newblock URL \url{https://doi.org/10.1142/S0129167X94000450}.

\bibitem[Lawson and Michelsohn(1989)]{Lawson}
H.~Blaine Lawson, Jr. and Marie-Louise Michelsohn.
\newblock \emph{Spin geometry}, volume~38 of \emph{Princeton Mathematical
  Series}.
\newblock Princeton University Press, Princeton, NJ, 1989.
\newblock ISBN 0-691-08542-0.

\bibitem[Lichnerowicz(1987)]{Lich2}
Andr\'{e} Lichnerowicz.
\newblock Spin manifolds, {K}illing spinors and universality of the {H}ijazi
  inequality.
\newblock \emph{Lett. Math. Phys.}, 13\penalty0 (4):\penalty0 331--344, 1987.
\newblock ISSN 0377-9017.
\newblock \doi{10.1007/BF00401162}.
\newblock URL \url{https://doi.org/10.1007/BF00401162}.

\bibitem[Lichnerowicz(1988)]{Lichnerowicz}
Andr\'{e} Lichnerowicz.
\newblock Killing spinors, twistor-spinors and {H}ijazi inequality.
\newblock \emph{J. Geom. Phys.}, 5\penalty0 (1):\penalty0 1--18, 1988.
\newblock ISSN 0393-0440.
\newblock \doi{10.1016/0393-0440(88)90011-3}.
\newblock URL \url{https://doi.org/10.1016/0393-0440(88)90011-3}.

\bibitem[Penrose(1967)]{Penrose1}
R.~Penrose.
\newblock Twistor algebra.
\newblock \emph{J. Mathematical Phys.}, 8:\penalty0 345--366, 1967.
\newblock ISSN 0022-2488.
\newblock \doi{10.1063/1.1705200}.
\newblock URL \url{https://doi.org/10.1063/1.1705200}.

\bibitem[Penrose(1975)]{Penrose2}
R~Penrose.
\newblock Twistor theory, its aims and achievements.
\newblock In R.~Penrose, D.~W. Sciama, and C.~J. Isham, editors, \emph{Quantum
  gravity}, pages 268--407. Clarendon Press, Oxford, 1975.
\newblock ISBN 0-19-851943-5.

\bibitem[Penrose and Rindler(1988)]{Penrose0}
Roger Penrose and Wolfgang Rindler.
\newblock \emph{Spinors and space-time. {V}ol. 2}.
\newblock Cambridge Monographs on Mathematical Physics. Cambridge University
  Press, Cambridge, second edition, 1988.
\newblock ISBN 0-521-34786-6.
\newblock Spinor and twistor methods in space-time geometry.

\end{thebibliography}

\end{document}